\documentclass[12pt]{amsart}
\usepackage{amsmath, amsthm, amssymb,cite,pgfplots}
\usepackage{fullpage}
\usepackage{hyperref}
\usepackage[color=yellow,disable]{todonotes}

\newtheorem{theorem}{Theorem}
\newtheorem{lemma}{Lemma}
\newtheorem{proposition}[lemma]{Proposition}
\newtheorem{corollary}[lemma]{Corollary}

\numberwithin{lemma}{section}

\numberwithin{equation}{section}

\newcommand{\R}{{\mathbb R}}
\newcommand{\C}{\mathbb C}

\newcommand{\Z}{{\mathbb Z}}

\newcommand{\hyp}{{\mathrm{hyp}}}
\newcommand{\el}{{\mathrm{ell}}}
\newcommand{\umod}{u_{\mathrm {mod}}}
\newcommand{\uscat}{u_{\mathrm{scatter}}}

\newcommand{\uj}{w}

\begin{document}

\title{The lifespan of small data solutions to the KP-I}

\author{Benjamin Harrop-Griffiths}
\address{Department of Mathematics, University of California at Berkeley}
\email{benhg@math.berkeley.edu}
 \thanks{The first author was partially supported by the NSF grant DMS-1266182}

\author{Mihaela Ifrim}
\address{Department of Mathematics, University of California at Berkeley}
\thanks{The second author was supported by the Simons Foundation}
\email{ifrim@math.berkeley.edu}

\author{ Daniel Tataru}
\address{Department of Mathematics, University of California at Berkeley}
 \thanks{The third author was partially supported by the NSF grant DMS-1266182
as well as by a Simons Investigator grant from the Simons Foundation}
\email{tataru@math.berkeley.edu}

\begin{abstract}
We show that for small, localized initial data there exists a global solution to the KP-I equation in a Galilean-invariant space using the method of testing by wave packets.
\end{abstract}

\maketitle


\section{Introduction}
In this paper we consider the Kadomtsev-Petviashvili equation (KP-I) initial-value problem
\begin{equation}\label{eq-1}
\begin{cases}
\partial_tu+\partial_x^3u-\partial_x^{-1}\partial_y^2u+\partial_x(u^2/2)=0\\
u(0,x,y)=u_0(x,y),
\end{cases}
\end{equation}
on $\R_t\times\R^2_{x,y}$. The KP-I equation and
the KP-II equation, in which the sign of the term
$\partial_x^{-1}\partial_y^2u$ in \eqref{eq-1} is $+$ instead of
$-$, were derived in \cite{KadomtsevPetviashvili} as models for the propagation of
dispersive long waves with weak transverse effects.

The Cauchy theory for \eqref{eq-1} has been extensively studied \cite{MR1635416,MR2047648,MR2719895,MR2415308,MR2097033,MR2358259,MR1944575,MR1933858,MR2976047}. In particular, \eqref{eq-1} is known to be locally well-posed \cite{MR2719895} in the anisotropic space \(H^{1,0}\) with 
\[
\|u\|_{H^{1,0}}^2 = \|u\|_{L^2}^2 +\|\partial_xu\|_{L^2}^2,
\]
and globally well-posed \cite{MR2415308} in the energy space \(\mathbf E^1\) where 
\[
\|u\|_{\mathbf E^1}^2 = \|u\|_{L^2}^2+\|\partial_xu\|_{L^2}^2+\|\partial_x^{-1}\partial_yu\|_{L^2}^2.
\]
Further work was devoted to the generalized KP equation, see for instance 
\cite{MR1685890,MR2830489}.

The question at hand is that of establishing global existence and asymptotics for solutions to \eqref{eq-1} with sufficiently small, regular and spatially localized initial data. To state our main result we begin with a discussion of the symmetries
of the equation \eqref{eq-1}:
\begin{enumerate}
\item \textit{Translation:} Translates of $u$ in $t$, $x$ and $y$ are solutions.
\item \textit{Reversal:} If \(u(t,x,y)\) is a solution, then so is
\(
u(-t,-x,\pm y)
\).
\item\textit{Scaling:} If $\lambda > 0$ then
\begin{equation}\label{eq:scaling}
u_\lambda(t,x,y) = \lambda^2 u(\lambda^3 t, \lambda x, \lambda^2 y)
\end{equation}
is also a solution.
\item\textit{Galilean invariance:}
For all $c\in \R$ the function
\begin{equation}\label{eq:galilei}
    u_c(t,x,y) = u(t,x-cy+c^2t,y-2ct)
\end{equation}
is a solution to \eqref{eq-1}.
Note that \(\hat u_c(t,\xi,\eta)=\hat u(t,\xi,\eta+c\xi)e^{-ic^2t\xi}e^{-2ict\eta}\).
\end{enumerate}

We denote by $\mathcal L$ the linear operator
\begin{equation}\label{ML}
\mathcal L:=\partial_t+\partial^3_x-\partial^{-1}_x\partial^2_y.
\end{equation}
To obtain pointwise estimates for solutions we introduce the ``vector fields"
\[
L_x:=x-3t\partial_x^2-t\partial^{-2}_x\partial^2_y,  \qquad L_y :=y+2t\partial^{-1}_x\partial_y,
\]
which commute with $\mathcal L$. We observe that $L_y \partial_x$ is the 
generator of the Galilean symmetry for both the linear and the nonlinear
equation. $L_x$, on the other hand, is not directly associated to a symmetry of the 
nonlinear equation. However, it arises in the expression  for the generator of the 
scaling symmetry, namely
\[
 S:=3t\partial_t+x\partial_x +2y\partial_y + 2 = 3t \mathcal L +  L_x \partial_x + 2 L_y \partial_y + 2. 
\]

In this article, following the the spirit of \cite{itNLS}, 
we seek to obtain a result which is Galilean-invariant. This is very natural,
as one should not expect the global well-posedness to depend on the reference frame of the observer.  For this reason, we will 
avoid using the scaling symmetry, as well as the use of conservation laws
which are not Galilean invariant (e.g. the energy). Instead, we will rely on 
the homogeneous scaling operator 
\[
S_0:= L_x \partial_x + L_y \partial_y.
\]
This commutes with both $\mathcal L$ and the Galilean group,  but is not associated to a symmetry of the nonlinear  equation. For large $t$ we will also use the following Galilean invariant operator
\[
L_z:=z+ 3t\partial_x^2, \qquad z:=-x+\frac{1}{4t}y^2,
\]
which relates to $S_0$ and $L_y$ as follows:
\[
L_z\partial_x= -S_0 +\frac{1}{4t}L_y^2\partial_x-\frac12.
\]
For \(z\geq0\), the symbol of \(L_z\) may be written as a product of the symbols of the operators
\[
L_z^\pm:=\sqrt{z}\pm i\sqrt{3t}\partial_x,
\]
which will be used repeatedly in our analysis. We note that \(L_z^+\) is hyperbolic on positive \(x\)-frequencies and elliptic on negative \(x\)-frequencies (and conversely for \(L_z^-\)).

For our global result we seek to use function spaces which are as simple as 
possible.  Toward that goal, we define the time-dependent space $X$ as
\[
\|u\|_{X}^2 = \|u\|_{L^2}^2 + \| u_{xxx}\|_{L^2}^2 
+\| L_y^2 \partial_x u\|_{L^2}^2 + \|S_0 u\|_{L^2}^2.
\]

Then, our main result is as follows:

\begin{theorem}
Assume that the initial data $u_0$ at time $0$ satisfies
\begin{equation}\label{data}
\|u_0\|_{X} \leq \epsilon \ll 1.
\end{equation}
Then, there exists a unique global solution $u$ which satisfies the bound
\begin{equation}\label{sln-energy}
\|u(t)\|_{X} \leq \epsilon \langle t \rangle ^{C \epsilon},
\end{equation}
as well as the pointwise bound 
\begin{equation}\label{sln-point}
\| u_x(t)\|_{L^\infty} \lesssim \epsilon t^{-\frac12} \langle t \rangle^{-\frac12}.
\end{equation}
Further, the solution $u$ scatters in $L^2$ at infinity, in the sense 
that there exists a linear wave $\uscat$  satisfying $\mathcal L \uscat = 0$ and 
\(
\|\uscat\|_{L^2} = \|u\|_{L^2}
\)
so that 
\begin{equation}\label{est:ScatteringBound}
\|u - \uscat\|_{L^2} \lesssim \epsilon^2 t^{-\frac{1}{48}+C\epsilon}. 
\end{equation}
\end{theorem}

To frame our result, we first note that the KP-I equation is integrable and admits a Lax pair representation. This leads to an infinite number of formally conserved quantities and allows solutions with small initial data to be studied using inverse scattering techniques (for example see the recent survey \cite{KSsurveyIST} and references therein).
\todo[inline]{
First few conservation laws are:
\[
\int u,\,\int (1+i \partial_x^{-1}\partial_y)u,\;\int u^2,\;\dots
\]
For conservation laws \(\#9\) and higher we have terms of the form \(\int(\partial_x^{-1}\partial_y(\phi^2))^2\) which is only well defined if \(\|u\|_{L^2}=0\).
}
\todo[inline]{Scattering problem is \(i\psi_x+\psi_{yy}+u\psi=0\). Can only be solved for small initial data.
}
However, it is of significant interest to develop more robust techniques to analyze the asymptotic behavior of solutions. In a recent paper Hayashi and Naumkin \cite{HayashiNaumkin2014} prove global existence and derive asymptotics for a certain class of rapidly decaying, smooth initial data. Our result presents a significant improvement by not only considering a larger class of initial data that includes the Schwartz functions, but also does so in a space that respects the Galilean invariance. Indeed, we believe this to be the only known global result for \eqref{eq-1} in a Galilean-invariant space. We note that our initial data space has norm
\[
\|u(0)\|_{X}^2 = \|u(0)\|_{L^2}^2 + \| u_{xxx}(0)\|_{L^2}^2 + \|y^2 u_x(0)\|_{L^2}^2 + \| (x \partial_x + y \partial_y) u(0)\|_{L^2}^2.
\]
\todo[inline]{Note Hayashi-Naumkin's data space does not contain \(\mathcal{S}\) as they assume \(\|\langle (x,y)\rangle^4\partial_x^{-1}u_0\|_{L^2}<\infty\).}

To describe the difficulties in this problem, we first note that the linear 
evolution $S(t)$ associated to the $\mathcal L$ operator exhibits $t^{-1}$
dispersive decay,
\[
\|S(t)\|_{L^1 \to L^\infty} \lesssim t^{-1}.
\]
This motivates the pointwise decay rate in \eqref{sln-point}. Unfortunately, this decay rate does not suffice in order to obtain uniform $X$ bounds for the nonlinear equation, and in turn close the bootstrap for the pointwise bound. This difficulty is a familiar one, and several methods have been used to bypass it in certain related problems.

The first such method is Shatah's normal form method \cite{MR803256}, which relies on the absence of bilinear resonant interactions in order to replace the quadratic nonlinearity with a cubic one. Unfortunately, our problem does admit three wave resonances. The symbol of $\mathcal L$ is 
\[
\ell (\tau,\textbf k)= \tau -\xi^3-\xi^{-1}\eta^2,  \qquad \textbf{k}:=(\xi, \eta).
\]
Hence the dispersion relation for \eqref{eq-1} is given by
\[
\omega(\textbf{k})=\xi^3+\xi^{-1}\eta^2.
\]
Thus, resonances in the bilinear interactions correspond to roots  of the system
\begin{equation*}
\left\{
\begin{aligned}
&\omega ({\textbf{k}_1})+\omega ({\textbf{k}_2})=\omega ({\textbf{k}_3})\\
&\textbf{k}_1+\textbf{k}_2=\textbf{k}_3,
\end{aligned}
\right.
\end{equation*}
which gives
\[
\frac{\eta_1}{\xi_1}-\frac{\eta_2}{\xi_2}=\pm\sqrt{3}(\xi_1+\xi_2).
\]
The presence of these three wave interactions prevents a classical normal form analysis: the quadratic nonlinearity is not removable on the set of resonances. Incidentally, we remark that this also prevents any attempts
to obtain global solutions via Strichartz type estimates or $X^{s,b}$ spaces.
We note, however, that for the closely related KP-II equation, where such 
resonant interactions do not occur, one can produce global solutions in this 
manner, precisely by employing the more robust $U^2$ and $V^2$ spaces, see \cite{MR2526409}.

More recently, a significant improvement over the normal form method was achieved 
with the space-time resonance method introduced by Germain-Masmoudi-Shatah \cite{MR2482120} and Gustafson-Nakanishi-Tsai \cite{MR2559713}, which was used to treat 
 a good number of two dimensional problems, e.g. \cite{MR2914945,MR2993751,GMS3dCapillary}. This essentially requires a weaker assumption, namely that there are no resonant interactions of parallel waves.  However, in our problem, waves of opposite frequencies are parallel and interact to yield resonant zero frequency output. Further, the symbol of $\mathcal L$ is also singular at zero $x$-frequency.

Instead of pursuing a Fourier based method as above,  our  result makes use of the \textit{method of testing by wave packets} \cite{itNLS,itWW,itCW,hgmKdV}, originally developed in the context of the \(1d\) cubic NLS \cite{itNLS} and \(2d\) water waves \cite{itWW,itCW}, and then applied to the mKdV
in \cite{hgmKdV}.  This relies on an even weaker nonresonance condition, namely that in resonant interactions  it is not possible to have all three waves travel in the same direction.  To describe this in more detail, consider the Hamiltonian flow corresponding to \eqref{eq-1}, which is given by
\begin{equation}\label{eq:HFlow}
\left\{
\begin{aligned}
&(x,y)\mapsto (x-3t\xi^2+t\xi^{-2}\eta^2,y-2t\xi^{-1}\eta)\\
&(\xi,\eta)\mapsto (\xi,\eta).
\end{aligned}
\right.
\end{equation}
In particular, for \(v_1,v_2\) satisfying \(v=-v_1+\frac{1}{4}v^2_2\geq0\), we expect solutions initially localized spatially near zero and in frequency near 
$\pm (\xi_v,\eta_v)$, where
\begin{equation}\label{xivetav}
(\xi_v,\eta_v) = \left(\frac{\sqrt{v}}{\sqrt{3}},-\frac{v_2\sqrt v}{2 \sqrt{3}}\right),
\end{equation}
to travel along the ray
\begin{equation}\label{trajectory}
\Gamma_{\mathbf{v}=(v_1,v_2)} := \{x = v_1 t, \ y=v_2t \}.
\end{equation}
This computation also directly leads to the phase function
\begin{equation}\label{phi-def}
\phi:=-\frac{2}{3\sqrt3}t^{-\frac12}\left(\frac{y^2}{4t}-x\right)^{\frac32}=- \frac{2}{3\sqrt3}t^{-\frac12}z^{\frac32},
\end{equation}
associated to the linear propagator $S(t)$. This satisfies $\nabla_{x,y} \phi(x,y) = (\xi_v,\eta_v)$, and also the eikonal equation $\ell(\nabla \phi) = 0$.  
We remark that the kernel of the linear propagator $S(t)$ will essentially have the form $t^{-1} \Re e^{i\phi}$ in the propagation region $\{v \geq 0\}$, with rapid decay away from it.

To conclude our discussion, we observe that, on the one hand, waves corresponding to different rays $\Gamma_{\mathbf v}$, $\Gamma_{\mathbf w}$ will have little interaction as they separate in the physical space. On the other hand,
waves corresponding to the same ray $\Gamma_{\mathbf v}$ have a significant interaction, but the frequency of this interaction will correspond to velocities which are away from $\mathbf v$.

We further comment on the scattering result, which is subtly different from standard
linear scattering. Precisely, we remark that, while $u$ approaches
the linear scatterer $\uscat$ in $L^2$, one property that fails
in this setting is the stronger bound $\mathcal L(u-\uscat) \in L^1 L^2$, or any other related Strichartz bound. To remedy this, we explicitly compute a quadratic correction
$\umod$, decaying in $L^2$, so that $\mathcal L(u-\uscat- \umod) \in L^1 L^2$.

A natural question in this setting is what is the regularity of the data $\uscat(0)$
for the scattering solution. One might expect that $\uscat(0) \in X(0)$, and we conjecture that this is indeed the case. However, our estimates only yield
the slightly weaker interpolation bound
\[
\uscat (0) \in [L^2,X(0)]_{C\epsilon},
\]
which is close to $X(0)$ but not quite there.

Our strategy of the proof will be to start with the pointwise bound 
\eqref{sln-point} as a bootstrap assumption. The goals of the 
 subsequent sections in the paper are as follows:
 
\begin{itemize}
\item{Energy estimates}, proved using the bootstrap assumption.

\item{Initial pointwise bounds}; these are obtained from the 
energy estimates using Klainerman-Sobolev type inequalities adapted 
to our problem.

\item{Final pointwise bounds}, closing the bootstrap argument using
the  the wave packet testing method.

\item{The scattering result}, whose proof relies on computing 
the quadratic correction $\umod$ mentioned above.
\end{itemize}

\subsection*{Acknowledgement} This research was carried out while the authors were visiting the Hausdorff Research Institute for Mathematics in Bonn.


\section{Energy estimates}

In this section we prove the energy estimates \eqref{sln-energy} under the bootstrap assumption 
\begin{equation}\label{est:bootstrap}
|u_x|\leq C \epsilon t^{-\frac12}\langle t\rangle^{-\frac12}.
\end{equation}
Precisely, we have:
\begin{proposition}\label{p:energy}
Let $u$ be a solution for the  KP-I equation in a time interval $[0,T]$,
so that 

(i) The initial data $u_0$ for the KP-I equation satisfies
\eqref{data}.

(ii) The solution $u$ satisfies \eqref{est:bootstrap}.

Then $u$ also satisfies the following energy estimate in $[0,T]$:
\begin{equation}\label{en}
\|u(t)\|_{X} \lesssim \epsilon \langle t \rangle ^{C^* \epsilon}, \qquad C^* \lesssim C.
\end{equation}
\end{proposition}

\begin{proof}
We first observe that the $L^2$ norm of the solution $\|u\|_{L^2}$
is a conserved quantity. Secondly, we note that we have good $L^2$ bounds
for the linearized equation 
\begin{equation}\label{eq-lin}
\begin{cases}
\partial_t w+\partial_x^3 w-\partial_x^{-1}\partial_y^2w +\partial_x(u w)=0\\
w(0,x)=w_0(x).
\end{cases}
\end{equation}
Indeed, we have 
\[
\frac{d}{dt} \|w\|_{L^2}^2 =2\int - w \partial_x (uw) dx =  -\int u_x w^2 dx \leq \|u_x\|_{L^\infty}\|w\|^2_{L^2}. 
\]
By Gronwall's inequality this yields energy bounds for $u_x$, $\partial_x L_y u$ and also $Su$ (not needed). 

The function $\partial_x^3 u $ solves a perturbed linear equation,
\[
\mathcal L (\partial_x^3 u) + \partial_x (u \partial_x^3 u) = - 3 u_{x}u_{xxx} - 3 (u_{xx})^2
\]
But the term on the right is bounded in $L^2$ by 
\[
\| u_{x}u_{xxx} + (u_{xx})^2\|_{L^2} \lesssim \|u_x\|_{L^\infty} \|u_{xxx}\|_{L^2} ,
\]
so the energy estimate for $u_{xxx}$ closes in the same way as in the case of the linearized
equation.


A similar argument applies for the energy bound for $\partial_x L_y^2 u$. To show this, we first prove the following interpolation inequality:

\smallskip
\begin{lemma}
For \(t\neq0\), we have the estimate
\begin{equation}
\|\partial_x L_y u\|_{L^4}^2\lesssim \|u_x\|_{L^\infty}\|\partial_xL_y^2u\|_{L^2}.\label{est:LySobolev}
\end{equation}
\end{lemma}
\begin{proof}
For \(t\neq0\) we write
\(
f(t,x,y) = u(t,x+\frac{1}{4t}y^2,y),
\)
and observe that
\[
\|\partial_xL_yu\|_{L^4} = 2t\|f_y\|_{L^4}.
\]

For dyadic \(\lambda\in 2^\Z\), we take the projection \(P_\lambda\) to act in the \(x\)-variable. Integrating by parts in the \(y\)-variable we obtain
\[
\|P_\lambda f_y\|_{L^4}^4 \leq 3\|P_\lambda f\|_{L^\infty}\|P_\lambda f_{yy}\|_{L^2}\|P_\lambda f_y\|_{L^4}^2\lesssim \|P_\lambda f_x\|_{L^\infty}\|P_\lambda \partial_x^{-1}f_{yy}\|_{L^2}\|P_\lambda f_y\|_{L^4}^2.
\]

Replacing \(L^4\) by the Lorentz space \(L^{4,4}\) and summing over dyadic \(x\)-frequencies using the Cauchy-Schwarz inequality, we have
\[
\|f_y\|_{L^4}^4 \sim \|f_y\|_{L^{4,4}_xL^4_y}^4\lesssim \|f_x\|_{L^\infty}\|\partial_x^{-1}f_{yy}\|_{L^2}\|f_y\|_{L^{4,4}_xL^4_y}^2\sim \|f_x\|_{L^\infty}\|\partial_x^{-1}f_{yy}\|_{L^2}\|f_y\|_{L^4}^2,
\]
from which the estimate \eqref{est:LySobolev} follows.
\end{proof}
\smallskip
%

We then consider the equation solved by \(\partial_xL_y^2u\),
\[
\mathcal L ( \partial_x L_y^2 u ) = -\partial_x L_y^2 \partial_x (u^2/2) =
- u \partial_x^2 L_y^2 u - (\partial_x L_y u)^2 .
\]
Thus, using the estimate \eqref{est:LySobolev} and integration by parts, we obtain
\[
\begin{split}
\frac{d}{dt} \|\partial_x L_y^2 u \|_{L^2}^2 = & \ \int - \partial_x L_y^2 u
(u \partial_x^2 L_y^2 u + (\partial_x L u)^2) dx 
= \int \left( \frac12 u_x |\partial_x L_y^2 u|^2 - \partial_x L_y^2 u (\partial_x L_y u)^2 \right) dx \\
\lesssim  & \ \|u_x\|_{L^\infty} \|\partial_x L_y^2 u\|^2_{L^2}.
\end{split}
\]
and conclude again via Gronwall's inequality.

It remains to obtain $L^2$ bounds for the expression $S_0 u$.
We do this differently for small $t$  and for large $t$.  For small $t$ it suffices to consider the following modification,
\[
w = S_0 u - t u u_x.
\]
The function $w$ also solves a perturbed linearized equation,
\[
\mathcal L w= - (uw)_x + 6 t u_{x}u_{xxx}.
\]
for which we directly obtain energy estimates by using the bootstrap
assumption \eqref{est:bootstrap} for $u_x$. By \eqref{est:bootstrap}
we can also estimate the size of the modification $t u u_x$ in $L^1 L^2$.
The above equation for $w$  is easily checked using the relations
\[
\mathcal L(uu_x)=(u\mathcal L u)_x+3(u_xu_{xx})_x-u_y^2+\partial_x^{-1}u_{yy}u_x,
\]
\[
(L_x\partial_x+L_y\partial_y)(uu_x)=(u(L_x\partial_x+L_y\partial_y)u)_x-9t(u_xu_{xx})_x+tu_y^2-t\partial_x^{-1}u_{yy}u_x-uu_x.
\]

For large $t$ we instead use the relation
\[
(L_x \partial_x + L_y \partial_y)u = S u - L_y \partial_y u - 3 t\mathcal L u -2u
\]
to reduce the problem to an estimate for $w = S u - L_y \partial_y u$.
Since $S$ is a generator of a symmetry for the system, it follows that $Su$ solves the linearized equation \eqref{eq-lin}. It remains to compute
\[
\begin{split}
\mathcal L (L_y \partial_y u) = & \ - \frac12 L_y \partial_y \partial_x( u^2 )
\\ =  & \  - \partial_x ( u L_y \partial_y u) - (L_y \partial_x u)( \partial_y u )
+ (L_y \partial_y u )(\partial_x u)
\\ = & \ - \partial_x ( u L_y \partial_y u) + \frac{1}{2t} [ (L_y \partial_x u)^2  
- (L_y^2 \partial_x u)( \partial_x u)] + uu_x.
\end{split}
\]
Thus we obtain 
\[
\mathcal L w =  - \partial_x ( u w) +\frac{1}{2t} [ (L_y \partial_x u)^2  
- (L_y^2 \partial_x u)( \partial_x u)] + uu_x.
\]
Hence the energy estimate for $w$ follows using \eqref{est:LySobolev} and the $L^2$ bound for $L_y^2 \partial_x u$.

\end{proof}


\section{Klainerman-Sobolev Estimates}\label{sect:KS}

In this section we prove pointwise bounds for \(u,u_x\). Ignoring the dependence of the energy estimates on \(t,\epsilon\), we assume that
\begin{equation}\label{est:NRG0}
\|u\|_X   \lesssim  1.
\end{equation}
The expression $S_0 u$ is somewhat cumbersome to use directly; instead we use $L_z$, for 
which we  have the energy estimate
\begin{equation}\label{est:NRG01}
\|L_z\partial_xu\|_{L^2}\lesssim 1, \qquad t \gtrsim 1.
\end{equation}
By a slight abuse of notation we will consider \(\mathbf v\) to be a function of \((t,x,y)\) in this section, defined via the ray \(\Gamma_{\mathbf v}\) of the Hamiltonian flow as \(\mathbf v = (t^{-1}x,t^{-1}y)\). In particular we will write \(v = t^{-1}z\). 

Parity considerations and the symbol singularity at $\xi = 0$ lead us to decompose
$u$ into positive and negative $x$-frequencies,
\[
u = u^+ + u^- = 2 \Re u^+,\qquad u^+ = \overline{u^-}.
\]
The $X$ norm bound commutes with this decomposition,
\[
\| u^+\|_{X} = \|u^-\|_X = \frac{1}{\sqrt{2}} \|u\|_{X}.
\]

We now divide \(u^+\) into a hyperbolic and an elliptic part. The corresponding decomposition of \(u^-\) follows by taking complex conjugates. We first fix a constant \(\delta>0\) and take an almost orthogonal decomposition in \(x\)-frequency adapted to the lattice \(2^{\delta\Z}\),
\[
u^+=\sum\limits_{\lambda\in 2^{\delta\Z}}u_\lambda^+.
\]
Here $\delta$ is a small universal constant, which we will only need in order to control the ``resolution" of our decomposition in Section \ref{sect:WP}. The implicit constant in our estimates will depend on $\delta$, but this has no impact on our analysis. This decomposition is compatible with the $X$-norm, in that
\[
\|u^+\|_{X}^2 \approx \sum_{\lambda} \|u^+_\lambda\|_X^2.
\]

For $t \geq 1$ we further  decompose $u^+$ into the hyperbolic and elliptic parts,
\[
u^{\hyp,+}=\sum\limits_\lambda u^{\hyp,+}_\lambda,\qquad u^{\el,+}=u^+-u^{\hyp,+},
\]
where, for \(\lambda\geq t^{-\frac13}\), we define
\[
u^{\hyp,+}_\lambda=\chi_\lambda^\hyp u_\lambda^+,\qquad u^{\el,+}_\lambda=u_\lambda^+-u^{\hyp,+}_\lambda,
\]
for a compactly supported, smooth function \(\chi_\lambda^\hyp(t,z)\) localized spatially in the hyperbolic region
\[
B^\hyp_\lambda=\{v =3\lambda^2(1+O(\delta))\},
\]
which corresponds to the frequencies associated to $u_\lambda$. We remark that here we prefer to take $\chi_\lambda^\hyp$ to have compact spatial support. Then, $u_\lambda^{\hyp,+}$ and $u_\lambda^{\el,+}$ are only
localized  at \(x\)-frequencies \(\lambda(1+O(\delta))\) modulo rapidly decaying tails. These tails have size $O((t^\frac13 \lambda)^{-N})$, and play a negligible role in our analysis.

We further note that, as defined above, the hyperbolic component $u^{\hyp}$ is supported in the region $\{ v \gtrsim t^{-\frac23}\}$,
and in particular sits outside the parabola $z = 0$. With the above decomposition of $u$, we can now state the pointwise bounds
on $u$ and $u_x$ as follows:


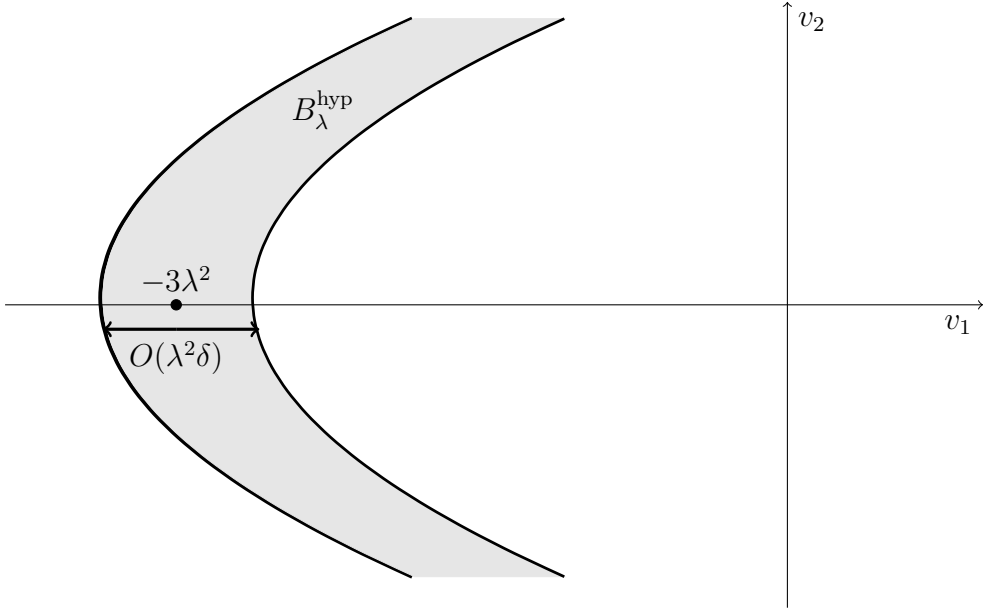
\begin{figure}
\caption{The hyperbolic region \(B^\hyp_\lambda\).}
\begin{tikzpicture}[scale=1.3]
\begin{scope}[rotate around={-90:(0,3.5)}]
\begin{axis}[stack plots=y,axis lines=none, no markers]
\addplot[black,smooth,thick] {x^2/4-1};
\addplot[black!10,fill=black!10,line width=0,smooth] {-3} \closedcycle;
\addplot[black,smooth,very thick] {-0.05};
\end{axis}
\end{scope}
\draw[->] (-4,0)--(6,0) node[below left] {\(v_1\)};
\draw[->] (4,-3.1)--(4,3.1) node[below right] {\(v_2\)};
\draw[very thick,<-] (-3,-0.25)--(-2.25,-0.25) node[below] {\(O(\lambda^2\delta)\)};
\filldraw(-2.25,0) circle (1.5pt) node[above] {\(-3\lambda^2\)};
\draw[very thick,->] (-2.25,-0.25)--(-1.4,-0.25);
\node at (-0.75,2) {\(B_\lambda^\hyp\)};
\end{tikzpicture}
\end{figure}


\vspace{0.3cm}
\begin{proposition}\label{propn:Pointwise}
For \(0<t<1\) we have the pointwise estimates
\begin{equation}\label{est:LessThan1}
|u|,|u_x|\lesssim t^{-\frac12}.
\end{equation}
For \(t\geq1\) we have the hyperbolic estimates
\begin{equation}\label{est:Hyperbolic}
\begin{aligned}
|u^\hyp|&\lesssim t^{-1}\min\{v^{-\frac34},v^{-\frac38}\},\\
|u^\hyp_x|&\lesssim t^{-1}\min\{v^{-\frac14},v^{\frac18}\},
\end{aligned}
\end{equation}
and the elliptic improvements
\begin{equation}\label{est:Elliptic}
\begin{aligned}
|u^\el|&\lesssim t^{-\frac34}\langle t^{\frac23}v\rangle^{-\frac34}(1+\log\langle t^{\frac23}v\rangle),\\
|u^\el_x|&\lesssim t^{-\frac{13}{12}}\langle t^{\frac23}v\rangle^{-\frac14}.
\end{aligned}
\end{equation}
\end{proposition}

\begin{proof}
It suffices to assume that \(\delta=1\) and prove bounds for \(u^+\). To simplify notation we drop the superscript and take \(u=u^+\). 

\medskip

\textbf{A. Small times $t \leq 1$.}
We recall the Sobolev estimate
\begin{equation}\label{est:UsefulSobolev}
|f|\lesssim\|f\|_{L^2}^{\frac14}\|\partial_xf\|_{L^2}^{\frac12}\|\partial_y^2f\|_{L^2}^{\frac14}.
\end{equation}
Taking \(f(t,x,y)=u_\lambda(t,x+\frac{1}{4t}y^2,y)\) in \eqref{est:UsefulSobolev}, we have
\[
|u_\lambda|\lesssim t^{-\frac12}\|u_\lambda\|_{L^2}^{\frac14}\|\partial_xu_\lambda\|_{L^2}^{\frac12}\|L_y^2\partial_x^2u_\lambda\|_{L^2}^{\frac14}.
\]
From the localization, we then have the estimate
\begin{equation}\label{est:SmallTimes}
|u_\lambda|\lesssim t^{-\frac12}\min\{\lambda^{-\frac32},\lambda^{\frac34}\}.
\end{equation}
For \(0<t<1\) we may sum with respect to  \(\lambda\) to get \eqref{est:LessThan1}.

\medskip

\textbf{B. Large times $t \geq 1$.} Here we split our analysis into low
frequencies and high frequencies, depending on the uncertainty principle
threshold $\lambda = t^{-\frac13}$ for the Airy type operator $L_z$.

\medskip

\textbf{B1. Low frequencies.} Here we consider times \(t\geq 1\) and frequencies $\lambda \leq t^{-\frac13}$. With $f_\lambda$ defined as above,
\[
f_\lambda(t,x,y)=u_\lambda(t,x+\frac{1}{4t}y^2,y)
\]
the $L^2$ bounds for $u$ and $L_z \partial_x u$ yield
\begin{equation}\label{b1}
\|f_{\lambda}\|_{L^2}\lesssim 1, \qquad \| x \partial_x f_{\lambda }\|_{L^2} \lesssim 1,
\end{equation}
while the bound on $L_y^2\partial_x u$ yields
\begin{equation}\label{b2}
\|\partial_y^2 f_{\lambda} \|_{L^2} \lesssim t^{-2} \lambda.
\end{equation}
We claim that the above three bounds imply the pointwise estimate
\begin{equation}\label{f-low}
|f_\lambda| \lesssim t^{-\frac34} \langle \lambda x \rangle^{-\frac34}.
\end{equation}
Summing this over $\lambda < t^{-\frac13}$ yields
\[
| f_{<t^{-\frac13}}| \lesssim t^{-\frac34}  \langle t^{-\frac13} x \rangle^{-\frac34} \log(1+ \langle t^{-\frac13} x \rangle).
\]
For $\partial_x f_\lambda$ we similarly obtain the same bound as \eqref{f-low} but with another factor of $\lambda$. Here summation 
over $\lambda < t^{-\frac13}$ is better than above, so we obtain 
\[
| \partial_x f_{<t^{-\frac13}}| \lesssim t^{-\frac{13}{12}}  \langle t^{-\frac13} x \rangle^{-\frac34}.
\]
Returning to $u$, the last two bounds imply, as desired, that
\begin{equation}
|u_{\leq t^{-\frac13}}|\lesssim t^{-\frac34}\langle t^{\frac23}v\rangle^{-\frac34}(1+\log\langle t^{\frac23}v\rangle),\qquad |\partial_x u_{\leq t^{-\frac13}}|\lesssim t^{-\frac{13}{12}}\langle t^{\frac23}v\rangle^{-\frac34}.\label{est:ReFix1}
\end{equation}

It remains to show that the bounds \eqref{b1} and \eqref{b2} imply \eqref{f-low}. We first observe that by scaling $(x,y) \to (\lambda x, \lambda^\frac12 y)$ the problem reduces to the case $\lambda = 1$.
For $f := f_1$ at frequency $1$ the bounds for $f$ and $\partial_x f$
are equivalent; precisely, from \eqref{b1} and \eqref{b2} one easily obtains the 
equivalent form
\begin{equation*}
\|\langle x \rangle f\|_{L^2}\lesssim 1, \qquad \| \langle x \rangle\partial_x f\|_{L^2} \lesssim 1,\qquad  \| \partial_y^2 f \|_{L^2} \lesssim t^{-2}.
\end{equation*}
At this point we localize $f$ to dyadic spatial regions $\langle x \rangle \approx r$, discarding the frequency localization. For $f_r = \chi_{\langle x \rangle \approx r} f$ we obtain
\begin{equation*}
\| f_r\|_{L^2}\lesssim r^{-1}, \qquad \| \langle x \rangle\partial_x f_r\|_{L^2} \lesssim r^{-1},\qquad \| \partial_y^2 f_r \|_{L^2} \lesssim t^{-2}.
\end{equation*}
Applying \eqref{est:UsefulSobolev} to $f_r$ yields
\[
|f_r| \lesssim t^{-\frac34} r^{-\frac34}
\]
and \eqref{f-low} follows.

\medskip

\textbf{B2. High frequencies.} Here we consider times \(t\geq 1\) and frequencies $\lambda \geq t^{-\frac13}$. The key step in the analysis is to carry out a careful 
analysis of the operator $L_z$, depending on the balance of $v=t^{-1}z$ and $\lambda$.
Precisely, in the region $v \approx \lambda^{2}$, which corresponds to $u^{\hyp}$
the operator $L_z$ is hyperbolic. Elsewhere, $L_z$ is elliptic.


\medskip
\begin{lemma} For \(t\geq1\) and \(\lambda\geq t^{-\frac13}\) we have the estimates
\begin{gather}
\|L_z^+u^{\hyp,+}_\lambda\|_{L^2}\lesssim \lambda^{-2}t^{-\frac12}     
(\|u_\lambda\|_{L^2} + \|L_z \partial_x u_\lambda\|_{L^2})
\label{est:LzHyp},\\
\|\langle \lambda^{-2}v\rangle u^\el_\lambda\|_{L^2}\lesssim\lambda^{-3}t^{-1}(\|u_\lambda\|_{L^2} + \|L_z \partial_x u_\lambda\|_{L^2})
\label{est:LzEll}.
\end{gather}
\end{lemma}
\begin{proof}
We remark that this is a one dimensional estimate, which applies for fixed $y$.
For simplicity we set $y=0$. By rescaling, it suffices to consider \(\lambda=1\). We note that the relation $\lambda > t^{-\frac13}$ is scale 
invariant, so after rescaling we still have $t > 1$.

Integrating by parts, we observe that for a smooth, compactly supported function \(f\),
\[
\|\sqrt vf\|_{L^2}^2+3\|\partial_xf\|_{L^2}^2=t^{-1}\|L_z^-f\|_{L^2}^2+2\sqrt3\int\sqrt{v}\Im(f\partial_x\overline f)\, dx.
\]
We apply this to \(f=L_z^+u_1^{\hyp,+}\).
Then $L_z^-f = L_z^-L_z^+u_1^{\hyp,+}$
can be directly estimated 
\[
\|L_z^- f\|_{L^2}\lesssim\|L_z u_1^+\|_{L^2}+\|u_1^+\|_{L^2}
\lesssim \|L_z\partial_x u_1\|_{L^2}+\|u_1\|_{L^2}
.
\]
where at the first step we use the spatial localization of the 
the cutoff function $\chi_1^{hyp}$ to the region $\{ z \approx |t|\}$
and at the second step we have used the localization 
of $u_1$ and thus $L_z u_1$ at frequency $1$.

On the other hand, $u_1$ is localized to positive 
unit frequencies. Hence $v^\frac14 L_z^+u_1^{\hyp,+}$ is also localized to positive frequencies modulo $O_{L_2}(t^{-N})$ errors.
Hence distributing the powers of $v$ we have
\[
\int\sqrt{v}\Im(L_z^+u_1^{\hyp,+}\partial_x\overline{(L_z^+u_1^{\hyp,+})})\, dx = \int\Im(v^{\frac14} L_z^+u_1^{\hyp,+}\partial_x\overline{(v^\frac14 L_z^+u_1^{\hyp,+})})\, dx 
\lesssim t^{-N} \|u_1\|_{L^2}^2.
\]
The estimate \eqref{est:LzHyp} then follows from the last two bounds.

For \eqref{est:LzEll} we will use the ellipticity of the operator $L_z$
in the support of $1-\chi_1^{hyp}$. For that we decompose
\[
u_1^\el=\chi_{\{|v|\ll1\}}u_1+\chi_{\{|v|\gg1\}}u_1+\chi_{\{v\sim-1\}}u_1
\]
for smooth cutoff functions \(\chi_{\{|v|\ll1\}},\chi_{\{|v|\gg1\}},\chi_{\{v\sim-1\}}\) localized to the corresponding regions.
Integrating by parts, we have the identity
\[
\|v\partial_xf\|_{L^2}^2+9\|\partial_x^3f\|_{L^2}^2=t^{-2}\|L_z\partial_xf\|_{L^2}^2+3\int v|\partial_x^2f|^2 \, dx.
\]
We then apply this for \(f=\chi_{\{|v|\ll1\}}u_1,\chi_{\{|v|\gg1\}}u_1,\chi_{\{v\sim-1\}}u_1\) respectively,  using 
Garding's inequality and the localization of $u_1$ to unit frequency in order to derive the estimates
\[
\begin{gathered}
\int v|\partial_x^2(\chi_{\{|v|\ll1\}}u_1)|^2 \, dx \ll\|\chi_{\{|v|\ll1\}}\partial_x^3 u_1\|_{L^2}^2 +  \|u_1\|_{L^2}^2, \\
\int v|\partial_x^2(\chi_{\{|v|\gg1\}}u_1)|^2 dx \ll\|\chi_{\{|v|\gg1\}} v \partial_x u_1\|_{L^2}^2
 + \|u_1\|_{L^2}^2
,\\
\int v|\partial_x^2(\chi_{\{v\sim-1\}}u_1)|^2 dx\leq 0.
\end{gathered}
\]
Then the bound \eqref{est:LzEll} follows.

\end{proof}


\medskip

\textbf{B2(a) The hyperbolic part.}
Applying \eqref{est:UsefulSobolev} with
\({f(t,x,y)=e^{-\frac{2}{3\sqrt3}it^{-\frac12}|x|^{\frac32}}u^\hyp_\lambda(t,x+\frac{1}{4t}y^2,y)}\)
we have
\[
|u^\hyp_\lambda|\lesssim t^{-\frac34}\|u^\hyp_\lambda\|_{L^2}^{\frac14}\|L_z^+u^\hyp_\lambda\|_{L^2}^{\frac12}\|L_y^2\partial_x^2u^\hyp_\lambda\|_{L^2}^{\frac14}.
\]
Hence by \eqref{est:LzHyp} we obtain
\[
|u_\lambda^\hyp|\lesssim t^{-1}\min\{\lambda^{-\frac34},\lambda^{-\frac32}\}.
\]
The pointwise estimate \eqref{est:Hyperbolic} then follows from the fact that the \(u^\hyp_\lambda\) are supported in the essentially disjoint regions \(\{v\approx 3\lambda^2\}\).

\medskip 

\textbf{B2(b) The elliptic part.}
Applying \eqref{est:UsefulSobolev} with \({f(t,x,y) = u_\lambda^\el(t,x+\frac{1}{4t}y^2,y)}\), we have
\[
|u_\lambda^\el|\lesssim t^{-\frac12}\|\partial_xu_\lambda^\el\|_{L^2}^{\frac34}\|L_y^2\partial_xu_\lambda^\el\|_{L^2}^{\frac14}.
\]

Estimating as for the low frequencies, we apply the elliptic estimate \eqref{est:LzEll} on dyadic \(v\)-intervals to obtain
\[
|u^\el_\lambda|\lesssim t^{-\frac54}\lambda^{-\frac32}\langle\lambda^{-2}v\rangle^{-\frac34}.
\]
We may then sum over \(\lambda\geq t^{-\frac13}\), to get
\begin{equation}
\sum\limits_\lambda|u^\el_{\lambda}|\lesssim t^{-\frac34}\langle t^{\frac23}v\rangle^{-\frac34}(1+\log\langle t^{\frac23}v\rangle),\qquad\sum\limits_{\lambda}|\partial_xu^\el_{\lambda}|\lesssim t^{-\frac{13}{12}}\langle t^{\frac23}v\rangle^{-\frac14}.\label{est:ReFix2}
\end{equation}
The estimate \eqref{est:Elliptic} then follows from \eqref{est:ReFix1} and \eqref{est:ReFix2}.
\end{proof}


As a consequence of the localization of \(u^\hyp\) and the compatibility of the localization with the \(X\)-norm, we have the following corollary:

\begin{corollary}\label{WavePacketHelper}For \(t>1\), we have the estimates
\begin{equation}
\|v^{\frac12}L_z^+\partial_xu^{\hyp,+}\|_{L^2}\lesssim t^{-\frac12}\|u\|_X,\qquad\|v^{-1}\partial_x^3L_y^2u^\hyp\|_{L^2}\lesssim \|u\|_X.
\end{equation}
\end{corollary}


\section{Wave packets}\label{sect:WP}

To study the global decay properties of solutions to \eqref{eq-1} we apply the same idea as in \cite{itNLS,itCW,itWW,hgmKdV}, which is to test the
solution $u$ with wave packets which travel along the Hamilton flow. 
Since we aim to prove uniform bounds on $u_x$, it is simpler to test $u_x$ rather than $u$.

A wave packet, in the context here, is an approximate solution to the
linear system, with $O(t^{-1})$ errors. Precisely, for each trajectory
$\Gamma_{\mathbf{v}=(v_1,v_2)}$ as in \eqref{trajectory}, we establish decay for $u_x$ along this ray by testing with a wave packet moving along the
ray with velocity $\mathbf{v}$ where, in contrast to Section~3, we now consider \(\mathbf v\) to be independent of \((t,x,y)\).

To motivate the definition of this packet we recall some
useful facts. First, this ray is associated with waves that have
spatial frequencies $\pm  (\xi_{v},\eta_v)$ as in \eqref{xivetav}. 
Thus, it is convenient to use the phase function $\pm \phi$, with $\phi$ 
as in \eqref{phi-def}. Selecting the $+$ sign, which corresponds to positive $x$-frequencies,  it is natural to use as test functions  wave  packets  of the form
\begin{equation}
\label{defPsi}
\Psi_{\mathbf v}(t,x,y)= - i\sqrt{3}v^{-\frac{1}{2}}\partial_{x}\left( \chi(\lambda_1(z-vt),\lambda_2(y-v_2t)) e^{i\phi (t,x,y)}\right),
\end{equation}
where
\[
\lambda_1=t^{-\frac12}v^{-\frac14},\qquad\lambda_2=t^{-\frac12}v^{\frac14}.
\]
 Here we  take $\chi$ smooth with compact support. For normalization purposes we assume that
\[
 \int\chi(\alpha, \beta)\, d\alpha d\beta = 1.
\]
The $t^\frac12$ localization scale is exactly the scale of wave packets which are required to stay coherent on the time scale $t$. The $v$ factors account for the 
different dispersion rates in the $x$ and the $y$ directions. Finally, the 
$x$ derivative is used in order to simplify the computation of $\mathcal L \Psi_{\mathbf v}$. For other purposes we note that the leading part of $\Psi_{\mathbf{v}}$ is given by
\[
\Psi_{\mathbf{v}}=\chi e^{i\phi}+ O(\lambda_1).
\]

To see that these are reasonable approximate solutions we
observe that we can compute
\begin{equation}
\label{error}
\begin{aligned}
e^{-i\phi}\mathcal{L}\Psi_{\mathbf{v}}=
&\frac12t^{-1}\left[ -\partial_x((z-vt)\chi )+(\partial_y+\frac{y}{2t}\partial_x)((y-v_2t)\chi)\right] \\
&-i\sqrt{3}\left( v^{-\frac{1}{2}}(\partial_y+\frac{y}{2t}\partial_x)^2\chi-v^{\frac12}\partial_x^2\chi\right) +O(t^{-\frac{3}{2}} v^{\frac{1}{4}}).
\end{aligned}
\end{equation}

The explicit terms above  are the leading ones, and, as expected, have size $t^{-1}$
times the size of $\Psi_{\mathbf{v}}$; further, they exhibit some additional
structure, manifested in the presence of the outer 
differentiation operators $(\partial_x, \partial_y)$, which we
will take advantage of later on.  The error term at the end has similar localization and regularity, but its size is better by another $t^{\frac12}$ factor, so no further structure information is needed.

The above computation shows that our wave packet $\Psi_v$ is  indeed an approximate
solution for  the linear equation in \eqref{eq-1}.  To be more precise,  as in \cite{itNLS, itCW, itWW,hgmKdV}, our test packet $\Psi_{\mathbf{v}}$ is a good approximate solution for the linear equation associated to our problem only on the dyadic time scale $\Delta t\ll t$. Nevertheless, we are using these packets as test functions 
globally in time, and this is where the extra structure above is relevant.

The outcome of testing solutions of \eqref{eq-1} with the wave packet $\Psi_{\mathbf{v}}$ is  the
scalar complex valued function $\gamma(t,\mathbf{v})$, defined in the region $\{v \geq  t^{ -\frac23}\}$ (this is the region where the $O(t^{-\frac{3}{2}} v^{\frac{1}{4}})$ terms in \eqref{error} can actually be treated as error terms):
\[
\gamma(t,\mathbf{v}) := \int u_x \bar \Psi_{\mathbf{v}} \, dxdy,
\]
which we will use as a good measure of the size of $u_x$ along
our chosen ray.

For the purpose of proving global existence of the solutions  we only need to consider $\gamma$ along a single ray. However, in order to obtain a more precise
asymptotics we will think of $\gamma$ as a function $\gamma(t,\mathbf{v})$.

The main purpose of the remaining part of this section is to establish qualitative properties for $\gamma$ and this will be done in the two propositions below. As a prerequisite, we need the following estimates:

\begin{lemma}
\label{puff} Assume $w : \R^2\rightarrow \C$ is a compactly supported function.  Then the following estimate holds whenever all factors on the right are finite:
\begin{equation}
\label{inegalitatea}
\| w\|_{\dot C^\frac14} \leq (\|w_x\|_{L^2}+\|w_y\|_{L^2})^{\frac14}
\Vert w_x\Vert_{L^2}^{\frac12}\| w_{yy}\|_{L^2}^\frac14.
\end{equation}
\end{lemma}

\begin{proof}
The proof of this lemma is fairly straightforward; we use the embedding of  $\dot{H}^{\frac{3}{4}}$ into $\dot{C}^{\frac14}$ and interpolation to obtain 
\[
\Vert w\Vert_{L^{\infty}_x\dot{C}^{\frac14}_y}\leq \Vert w\Vert_{L^{\infty}_x\dot{H}^{\frac34}_y} \lesssim \|w\|^{\frac12}_{L^2_x \dot H^\frac32_y}
\|w_{x}\|^{\frac12}_{L^2_x L^2_y} \lesssim \|w_y\|_{L^2}^{\frac14}
\Vert w_{yy}\Vert_{L^2}^{\frac14}\| w_{x}\|_{L^2}^\frac12.
\]
On the other hand, exchanging the order of the variables, we similarly have
\begin{equation*}
\label{e2}
\| w\|_{L^\infty_y\dot C^\frac14_x} \lesssim
\| w\|_{L^\infty_y \dot H^\frac34_x} \lesssim \|w\|^{\frac34}_{L^2_y \dot H^1_x}
\|w_{yy}\|^{\frac14}_{L^2_y L^2_x}=\|w_x\|^{\frac34}_{L^2}
\|w_{yy}\|^{\frac14}_{L^2 }.
\end{equation*}
The two bounds above complete the proof of \eqref{inegalitatea}.
\end{proof}
Now we are left with two tasks. Firstly, we need to show that $\gamma$ is a good representation of the
pointwise size of $u_x$, and for this we need to  compare $u_x$ to $\gamma(t,\mathbf{v})$ as follows:
\begin{proposition}
\label{lema1}
The function $\gamma$ satisfies the uniform bound
\begin{equation}\label{bd-gamma}
\begin{split}
\|\gamma\|_{L^\infty} \lesssim t \|u_x\|_{L^\infty},
\end{split}
\end{equation}
as well as the approximation error estimate
\begin{equation}
\label{diff-x}
\begin{aligned}
 \| u_x(t,\mathbf{v} t) - 2t^{-1} \Re\{e^{i\phi(t,\mathbf{v}t)} \gamma(t,\mathbf{v})\}\|_{L^{\infty}} \lesssim  &v^{-\frac{1}{16}}t^{-\frac{9}{8}}\Vert u\Vert_X.
 \end{aligned}
\end{equation}
\end{proposition}
\begin{proof}
The first estimate \eqref{bd-gamma} is straightforward as 
\[
\int |\Psi_{\mathbf{v}}| dx dy = t.
\]

We turn our attention to \eqref{diff-x}, where we will take advantage of 
the dyadic decomposition of $u_x$ in Section~\ref{sect:KS}.  

We first observe that we can simplify the
expression of $\Psi_{\mathbf{v}}$ in the formula for $\gamma$: the lower order terms in $\Psi_{\mathbf{v}}$ are better by a factor of $v^\frac14 t^{-\frac12}$, therefore we can readily replace $\Psi_{\mathbf{v}}$ by $\chi e^{i\phi}$. Thus, it suffices to work with
 \begin{equation}
 \label{new-gamma}
 \gamma(t,\mathbf{v})=\int u_x e^{-i\phi} \chi \, dxdy.
 \end{equation}
 
 Decomposing $u_x=u_x^{\el}+u_x^{\hyp, -}+u_x^{\hyp, +}$ we observe that only the last term has a nontrivial contribution to $\gamma$
 \begin{equation}
 \label{new-new-gamma}
 \gamma =\int u_{x}^{\hyp, +}e^{-i\phi}\chi\, dxdy+ O(t^{-N}),
 \end{equation}
where $N $ is arbitrarily large. The contributions of $u_x^{\hyp, -}$  and $u^{\el}$ decay as $t^{-N}$ since there are no resonant frequency interactions ($u_x^{\hyp,-}$ and $u_x^{\el}$ are frequency localized away from $(\xi_v, \eta_v)$ where $\Psi_{\mathbf{v}}$ is localized). We can further harmlessly replace $ u_{x}^{\hyp, +}$  by its component $ u_{v,x}^{\hyp, +}$ associated to the dyadic frequency 
associated to $\xi_v$ (and its immediate neighbors).

Now we turn our attention to $u_x$ in \eqref{diff-x}. For the elliptic part  $u_x^{\el}$ we already have a satisfactory estimate in \eqref{est:Elliptic}. On the other hand $u_x^{\hyp}=2\Re u_{x}^{\hyp,+}$. Hence, given the above considerations  it suffices to estimate the difference 
\[
\mathcal{D}:= u^{\hyp, +}_{v,x} (t, \mathbf{v}t)-t^{-1}\int u_{v,x}^{\hyp, +}e^{-i\phi}\chi\, dxdy.
\]
 Introducing the notation  $w(t,z,y):=e^{-i\phi(t,x,y)}u^{\hyp, +}_{v,x}(t, x,y)$ we compute 
 \begin{equation*}
\begin{aligned}
 e^{-i\phi(t,\mathbf{v}t)} \mathcal{D} &=  w(t,v t, v_2t) -t^{-1} \int  w(t, z,y) \chi \, dzdy\\
&= t ^{-1}\int \left[ w(t,v t, v_2t) - w(t,z,y)\right]   \chi(\lambda_1(z-vt),\lambda_2(y-v_2t)) \, dzdy\\
&=-\int  \left[ w(t,v t, v_2t) - w(t,\lambda_1^{-1}\alpha+vt,\lambda_2^{-1}\beta+v_2t)\right]  \chi(\alpha,\beta) \, d\alpha d\beta\\
&=-\int  \left[ \tilde{w}(t,0, 0) - \tilde{w}(t,\alpha, \beta)\right]  \chi(\alpha,\beta) \, d\alpha d\beta ,\\
\end{aligned}
\end{equation*}
where we have used the notation $\tilde{w}(t,\alpha,\beta):=w(t,\lambda_1^{-1}\alpha+vt,\lambda_2^{-1}\beta+v_2t)$. Hence,
 \begin{equation}
 \label{diferenta}
\begin{aligned}
&\vert  \mathcal{D} \vert \lesssim \int  \vert \tilde{w}(t,0, 0) - \tilde{w}(t,\alpha, \beta)\vert  \chi(\alpha,\beta) \, d\alpha d\beta.\\
\end{aligned}
\end{equation}

To estimate the RHS above we use the bound in Lemma~\ref{puff}.  This gives
 \begin{equation}
\begin{aligned}
\label{floarea}
 \vert \tilde{w}(t,0, 0) - \tilde{w}(t,\alpha, \beta)\vert  &\lesssim (|\alpha|+|\beta|)^\frac14 \left(  \Vert \tilde{w}_{\alpha}\Vert_{L^2}+ \Vert \tilde{w}_{\beta}\Vert_{L^2}\right)^{\frac34}\Vert  \tilde{w}_{\beta\beta}\Vert^{\frac{1}{4}}_{L^2}.
 \end{aligned}
\end{equation}

 To conclude the bound for $\mathcal{D}$ it remains to  reinterpret the result of \eqref{floarea} in terms of the original function $u^{\hyp, +}_{v,x}(t,x,y)$. For that we compute
\begin{equation*}
\begin{aligned}
&\tilde{w}_{\alpha}(t,\alpha,\beta)=\lambda_1^{-1}\partial_z w(t,z,y)
=\frac{-ie^{i\phi (t,x,y)}v^{\frac{1}{4}}}{\sqrt{3}} L^{+}_z  u_{v,x}^{\hyp, +}(t,x,y),\\
&\tilde{w}_{\beta}(t,\alpha,\beta)=\lambda_2^{-1}\partial_{y} w(t,z,y)=\frac{e^{-i\phi}v^{-\frac{1}{4}}}{2t^{\frac{1}{2}}}\left[L_y\partial_x^2  u_{v}^{\hyp, +}\right],\\
&\tilde{w}_{\beta\beta}(t,\alpha,\beta)=\lambda_2^{-2}\partial_{yy} w(t,z,y)= \frac{e^{-i\phi}v^{-\frac{1}{2}}}{4t}\left[  L^2_y\partial^2_x u_{v,x}^{\hyp, +}\right],\\
\end{aligned}
\end{equation*}
and the corresponding $L^2$ norms in the initial variable:
\begin{equation*}
\begin{aligned}
 \Vert \tilde{w}_{\alpha}\Vert_{L^2_{\alpha\beta}}=&\frac{v^{\frac14}}{\sqrt{3}}t^{-\frac{1}{2}}\Vert L^{+}_z  \partial_x u_{v}^{\hyp, +}\Vert_{L^2_{xy}}\lesssim v^{-\frac14}t^{-1}\Vert u\Vert_{X},\\  \Vert \tilde{w}_{\beta}\Vert_{L^2_{\alpha \beta}}=&\frac{v^{-\frac14}}{2}t^{-1}\Vert  L_y\partial_x^2 u_{v}^{\hyp, +}\Vert_{L^2_{xy}}\lesssim v^{-\frac{1}{4}}t^{-1}\Vert u\Vert_{X},\\
  \Vert \tilde{w}_{\beta \beta}\Vert_{L^2_{\alpha\beta}}=&\frac{v^{-\frac12}}{4}t^{-\frac32}\Vert L^2_y \partial_x^2 u_{v,x}^{\hyp, +}\Vert_{L^2_{xy}}\lesssim v^{\frac{1}{2}}t^{-\frac{3}{2}}\Vert u\Vert_X,\\
\end{aligned}
\end{equation*}
where we have used the bounds for $u^{\hyp,+}$ in Corollary~\ref{WavePacketHelper}.  Thus from \eqref{floarea} we obtain
\[
\vert \tilde{w}(t,0, 0) - \tilde{w}(t,\alpha, \beta)\vert  \lesssim 
v^{-\frac{1}{16}} t^{-\frac98} \|u\|_{X}, \qquad
|\alpha|+|\beta| \lesssim 1,
\]
which leads to a similar bound for $\mathcal D$. We remark that we can rewrite
this bound in terms of $w$ as
\begin{equation}\label{bdiff}
 |w(t,v t, v_2t) - w(t,z,y)|\lesssim 
v^{-\frac{1}{16}} t^{-\frac98} \|u\|_X, \qquad (t,z,y) \in \text{supp} \ \Psi_{\mathbf v},
\end{equation}
which will be useful later.

\end{proof}

Secondly, we need to show that $\gamma$ stays bounded,  which we do by  establishing a differential equation for it:

\begin{proposition}
\label{lema2}
If $u$ solves \eqref{eq-1}, then we have that 
\begin{equation}
\label{gamma-dot}
\dot{\gamma}(t, \mathbf{v})=O(t^{-\frac{13}{12}})(\|u\|_X+\|u\|_X^2 ), \quad v>t^{-\frac{1}{3}}.
\end{equation} 
\end{proposition}

\begin{proof}
 We obtain the differential equation for $\gamma$ by simply testing  \eqref{eq-1} against our wave packet $\Psi_{\mathbf{v}}$,
\begin{equation}
\label{gdot}
\dot{\gamma}(t,\mathbf{v})=\int  \mathcal{L}\bar{\Psi}_{\mathbf{v}}u_x-\bar{\Psi}_{\mathbf{v}}\partial_x(uu_x)\, dxdy.
\end{equation}
First we measure the error in the linear component of \eqref{gdot}.  We separate $u$ and $u_x$ into hyperbolic and elliptic parts. The  decay is slightly better in the elliptic case
\[
 \left| \int  \mathcal{L}\bar{\Psi}_{\mathbf{v}}u^{\el}_x\, dx dy\right| \lesssim v^{-\frac14}t^{-\frac{5}{4}}\| u\|_X.
\]
For $u_x^{\hyp}$, we further decompose into $u_x^{\hyp, +}$ and $u_x^{\hyp, -}$. The contribution of the last one is of $O(t^{-N})$, due to mismatched frequencies. So we are left with the contribution of $u_x^{\hyp, +}$. For $\mathcal{L}\Psi_{\mathbf{v}}$ it suffices to consider its leading term from \eqref{error}, which is of order $O(t^{-1})$. This yields the following integral
\begin{equation*}
\begin{aligned}
 &\int\left[ \frac12t^{-1}\left[ -\partial_x((z-vt)\chi )+(\partial_y+\frac{y}{2t}\partial_x)((y-v_2t)\chi)\right]\right.  \\
 &\hspace*{1.5cm}\left.  +i\sqrt{3}\left( v^{-\frac{1}{2}}(\partial_y+\frac{y}{2t}\partial_x)^2\chi-v^{\frac12}\partial_x^2\chi\right) \right]e^{-i\phi} u^{\hyp, +}_x\, dx dy.
 \end{aligned}
\end{equation*}
 Using the bound \eqref{bdiff} for $ w(t,z,y):=e^{-i\phi(t,x,y)} u^{\hyp, +}_x(t,x,y)$, we approximate 
 \[
 w(t,z,y)=w(t,vt, v_2t)+O(v^{-\frac{1}{16}}t^{-\frac98})\|u\|_{X},
 \]
 and substitute it in the integral above. The contribution of the error term yields a $v^{-\frac{1}{16}}t^{-\frac98}$ bound, and the contribution of the leading term vanishes when we integrate by parts.

For the second term in \eqref{gdot} we integrate by parts and separate $u$ and $u_x$ into hyperbolic and elliptic parts. For example, when we estimate the hyperbolic and elliptic interaction, we make use of bounds obtained in \eqref{est:Hyperbolic} and \eqref{est:Elliptic} 
\begin{equation*}
\left| \int \bar{\Psi}_{\mathbf{v}}\partial_x(u^{\hyp}u_x^{\el})\, dxdy\right| \lesssim v^{-\frac{1}{2}}t^{-\frac{5}{4}}\Vert u\Vert_X^2.
\end{equation*}
The same argument applies whenever one of the factors is elliptic; so we are left only with the hyperbolic-hyperbolic interaction
\begin{equation*}
\int \bar{\Psi}_{\mathbf{v}}\partial_x(u^{\hyp}u_x^{\hyp})\, dxdy.
\end{equation*}
By definition, the hyperbolic components are frequency localized near $\pm(\xi_{v}, \eta_v)$, while $\bar{\Psi}_{\mathbf{v}}$ is localized at $-(\xi_v, \eta_v)$. Since the three interacting frequencies cannot add up to zero, it follows that the above integral is rapidly decreasing, i.e., is of order $\epsilon t^{-N}$, for $N$ large enough.
\end{proof}

In the last part of this section we finalize the bootstrap argument and prove \eqref{sln-point}. We already have the estimate for \(0<t<1\), so we consider
a time interval $[1,T]$ where we make the bootstrap assumption  
\[
\vert u_x\vert \leq C\epsilon t^{-1},
\]
with a fixed large universal constant $C$. Here $C$ is chosen with the property that 
\[
1\ll C \ll \epsilon^{-\frac12}.
\]
Under this assumption, by Proposition~\ref{p:energy}, $u$ satisfies the energy estimate \eqref{en} in the same time interval $ [1,T]$. From Proposition~ \ref{propn:Pointwise} we have 
\begin{equation}\label{abc}
\vert u_x\vert \lesssim \epsilon t^{-1+C_*\epsilon}\left[\min\{|v|^{-\frac14},|v|^{\frac18}\}+ t^{-\frac{1}{12}}\right].
\end{equation}
This immediately implies 
\begin{equation}\label{abcd}
\vert \gamma(t,\mathbf{v})\vert \lesssim \epsilon t^{C_*\epsilon}\min\{|v|^{-\frac14},|v|^{\frac18}\}.
\end{equation}
Our goal is to prove \eqref{sln-point}. For this we consider the domain $\Omega$
\[
\Omega:= \left\{ v \, : t^{-\alpha}\leq v\leq t^{\alpha}\right\},
\]
where $\alpha $ is a sufficiently small parameter; $\alpha =\frac{1}{6}$ suffices.
Outside $\Omega$, the bound in \eqref{sln-point} follows from \eqref{abc}. Inside $\Omega$ we use
use Proposition~ \ref{lema1} and Proposition~\ref{lema2} to show that  \eqref{sln-point} holds  with an implicit constant which does not depend on $C$.  From the bound \eqref{diff-x} obtained in Proposition~\ref{lema1}, and \eqref{en} we get
\begin{equation*}
 \| u_x(t,\mathbf{v} t) - 2t^{-1} \Re\{e^{i\phi(t,\mathbf{v}t)} \gamma(t,\mathbf{v})\}\|_{L^{\infty}} \lesssim  v^{-\frac{1}{16}}t^{-\frac{9}{8}}\Vert u\Vert_X\lesssim \epsilon v^{-\frac{1}{16}}t^{-\frac{9}{8}+\epsilon C_*}.
\end{equation*}
This estimate implies that inside $\Omega$ we can substitute  the bound \eqref{sln-point} for $u_x$ with its analogue for $\gamma$, namely 
\begin{equation}
\label{need}
|\gamma(t,\mathbf{v})| \lesssim \epsilon .
\end{equation}

Our goal now is to use the ODE \eqref{gamma-dot} in order to transition from 
\eqref{abcd} to \eqref{need} along rays $\Gamma_{\mathbf{v}}$. From Proposition~\ref{lema2} we have the following bound in $\Omega$
\begin{equation}
\label{gd}
\vert \dot{\gamma}(t, \mathbf{v})\vert \lesssim \epsilon t^{-\frac{13}{12}+2\epsilon C_*}.
\end{equation}

We consider three cases for $v$:

(i)  Suppose first that
$v \approx 1$, i.e., $z\approx t$. Then we initially have 
\[
|\gamma(t,\mathbf{v})| \lesssim \epsilon, \qquad t \approx 1.
 \]
Integrating \eqref{gamma-dot} we conclude that 
\[
|\gamma(t,\mathbf{v})| \lesssim \epsilon, \qquad t \geq 1,
\]
and \eqref{need} follows.

(ii) Assume now that $v \ll 1$, i.e., $z\ll t$. Then, as $t$ increases, the ray $\Gamma_{\mathbf{v}}$
enters $\Omega$ at some point $t_0$ with $v \approx t_0^{-\alpha}$. Then by 
 \eqref{abcd} we obtain
\[
|\gamma(t_0,\mathbf{v})| \lesssim \epsilon t_0^{\epsilon C_*}v^{\frac{1}{8}}\lesssim \epsilon.
\] 
We use this to initialize $\gamma$. For larger $t$ we use \eqref{gd}
to conclude that 
\[
|\gamma(t, \mathbf{v})| \lesssim \vert \gamma(t_0,\mathbf{v})\vert + \epsilon t_0^{-\frac{1}{12}+2\epsilon C_*} \lesssim \epsilon, \qquad t > t_0.
\]
Then  \eqref{need} follows.

(iii) Finally, consider the case $v \gg 1$, i.e., $z\gg t$.
Again, as $t$ increases, the ray $z = vt$ enters $\Omega$ at some point $t_0$ 
 $v \approx t_0^{\alpha}$, therefore by \eqref{abcd} we obtain
\[
|\gamma(t_0,\mathbf{v})| \lesssim  \epsilon t_0^{\epsilon C_*}v^{-\frac{1}{4}}  \lesssim \epsilon. 
\] 
We use this to initialize $\gamma$. For larger $t$ we use \eqref{gd}
to conclude that 
\[
|\gamma(t,\mathbf{v})| \lesssim  \vert \gamma(t_0,\mathbf{v})\vert + \epsilon t_0^{-\frac{1}{12}+2\epsilon C_*}  \lesssim \epsilon ,  \qquad t > t_0.
\]
Then  \eqref{need} again follows.


\section{Scattering}

In this section we prove the scattering estimate \eqref{est:ScatteringBound}.
We will use what we have already proved so far, namely  that we have a global solution $u$ which satisfies the bounds \eqref{sln-energy} and \eqref{sln-point}.

For fixed \(\alpha>0\), let
\[
\uj=P_{t^{-\frac\alpha2}\leq\cdot\leq t^{\frac\alpha2}}u,
\]
where the projection acts on \(x\)-frequencies. If \(\Omega\) is defined as in Section~\ref{sect:WP}, then for fixed \(\mathbf{v}\), the ray \(\Gamma_{\mathbf{v}}\) will eventually lie in \(\Omega\) and hence \(\uj^\hyp\) will capture the hyperbolic part of \(u\) at infinity. For concreteness we take \(\alpha=\frac16\) although any sufficiently small \(\alpha>0\) will suffice.
%

\medskip

\begin{lemma}
Let \(t\geq1\). We then have the estimate
\begin{equation}
\|uu_x-2\Re(\uj^+\uj^+_x)\|_{L^2}\lesssim \epsilon^2t^{-\frac{49}{48}+C\epsilon}\label{est:ScatHelper}.
\end{equation}
\end{lemma}
\begin{proof}
Let \(\chi\) be a bump function, identically \(1\) on the region \(\Omega\).
From \eqref{est:Hyperbolic} and \eqref{est:Elliptic} we have the estimate
\[
\|(1-\chi)uu_x\|_{L^2}+\|(1-\chi)\uj\uj_x\|_{L^2}\lesssim \epsilon^2t^{-\frac{49}{48}+C\epsilon}.
\]
On the other hand, from the localization of \(u^\hyp\), we have \(\chi(u^\hyp-\uj^\hyp)\), \(\chi(u^\hyp-\uj^\hyp)_x=0\), and from the elliptic estimate \eqref{est:Elliptic},
\[
\|\chi(u^\el-\uj^\el)\|_{L^\infty}+\|\chi(u^\el_x-\uj^\el_x)\|_{L^\infty}\lesssim\|\chi u^\el\|_{L^\infty}+\|\chi u^\el_x\|_{L^\infty}\lesssim \epsilon t^{-\frac98+C\epsilon}(1+\log t).
\]
Thus we are left with estimating the difference arising for the hyperbolic
parts, for which, using $u^{hyp} = 2 \Re u^{hyp,+}$, we write
\[
\chi(u^{hyp}u^{hyp}_x-2\Re(\uj^{hyp,+}\uj^{hyp,+}_x))
= \chi(u^{hyp}u^{hyp}_x-2\Re(u^{hyp,+} u^{hyp,+}_x))
= \chi \partial_x |u^{hyp,+}|^2
\]
To bound this last term, we use \eqref{est:LzHyp} to obtain
\[
\|\partial_x(|\uj^{\hyp,+}|^2)\|_{L^2}\lesssim t^{-\frac12}\|\uj^{\hyp,+}\overline{L_z^+\uj^{\hyp,+}}\|_{L^2}\lesssim\epsilon^2 t^{-\frac74+C\epsilon}.
\]
\end{proof}

\medskip


As \(2\Re(\uj^+\uj^+_x)\not\in L^1L^2\), we look to find an approximate solution to the equation
\[
\mathcal{L}\umod\approx 2\Re(\uj^+\uj^+_x)
\]
with error terms in \(L^1L^2\). As the bulk of $u^{\hyp,\pm}$, and thus of $\uj$, is localized at frequency $\pm ( \omega(\xi_v,\eta_v), \xi_v,\eta_v)$, a standard frequency/modulation analysis leads to the choice of a quadratic correction term
%
%
\begin{equation}\label{defn:umod}
\umod=\frac83\partial_x^{-3}\Re(\uj^+\uj^+_x).
\end{equation}
%
%

We observe that \(w^+w^+_x\) is localized at \(x\)-frequencies \(t^{-\frac{1}{12}}\lesssim\xi\lesssim t^{\frac{1}{12}}\), so \eqref{defn:umod} is well-defined and satisfies the \(L^2\) bound
\begin{equation}\label{est:umodBound}
\|\umod\|_{L^2}\lesssim\epsilon^2 t^{-\frac34}.
\end{equation}

The scattering bound \eqref{est:ScatteringBound} is then a consequence of the following Lemma.

\medskip


\begin{lemma}
For \(t\geq1\) and \(\epsilon>0\) sufficiently small, we have the estimate
\begin{equation}\label{est:ModScat}
\left\|2\Re(\uj^+\uj^+_x)-\mathcal{L}\umod\right\|_{L^2}\lesssim \epsilon^2 t^{-\frac74+C\epsilon}.
\end{equation}
\end{lemma}


\begin{proof}
We calculate
\[
\begin{split}
\mathcal{L}\partial_x(fg)&=\partial_x\left((\mathcal Lf)g+f(\mathcal Lg)+3f_{xx}g_x+3f_xg_{xx}-\frac{1}{2t}fg\right)\\
&\quad+\frac{1}{4t^2}(L_y^2\partial_xf)g_x+\frac{1}{4t^2}f_x(L_y^2\partial_xg)-\frac{1}{2t^2}(L_y\partial_xf)(L_y\partial_xg),
\end{split}
\]
which gives
\[
\begin{split}
\frac43\mathcal L\partial_x^{-3}(\uj^+\uj^+_x)&=\partial_x^{-2}\left(\frac43\uj^+[\partial_t,P^+_{t^{-\frac{1}{12}}\leq\cdot\leq t^{\frac{1}{12}}}]u+\frac23\uj^+ P^+_{t^{-\frac{1}{12}}\leq\cdot\leq t^{\frac{1}{12}}}(u^2)_x+4\uj^+_{xx}\uj^+_x-\frac{1}{3t}(\uj^+)^2\right)\\
&\quad+\partial_x^{-3}\left(\frac{1}{3t^2}(L_y^2\partial_x\uj^+)\uj^+_x-\frac{1}{3t^2}(L_y\partial_x\uj^+)^2\right).
\end{split}
\]
We observe that
\[
[P^+_{t^{-\frac{1}{12}}\leq\cdot\leq t^{\frac{1}{12}}},\partial_t]=t^{-1}P^+_{t^{\frac{1}{12}}}+t^{-1}P^+_{t^{-\frac{1}{12}}}, \qquad [P^+_{t^{-\frac{1}{12}}\leq\cdot\leq t^{\frac{1}{12}}},L_y]=0.
\]
For sufficiently small \(\epsilon>0\), the frequency localization, \eqref{sln-energy} and \eqref{sln-point} then yield the estimate
\[
\left\|\mathcal L\umod-4\partial_x^{-2}\left(\uj^+_x\uj^+_{xx}\right)\right\|_{L^2}\lesssim\epsilon^2(1+\epsilon)t^{-\frac74}.
\]

We may write
\[
\partial_x^2(\uj^+\uj^+_x)=4\uj^+_x\uj^+_{xx}+\frac{1}{3t}\uj^+_xL_z\uj^+-\frac{1}{3t}\uj^+ L_z\partial_x\uj^+.
\]
As
\[
[P^+_{t^{-\frac{1}{12}}\leq\cdot\leq t^{\frac{1}{12}}},L_z]=t^{-\frac{1}{12}}P^+_{t^{\frac{1}{12}}}+t^{\frac{1}{12}}P^+_{t^{-\frac{1}{12}}},
\]
we may commute the frequency localization with \(L_z\) and estimate as before to get
\[
\left\|\uj^+\uj^+_x-4\partial_x^{-2}\left(\uj^+_x\uj^+_{xx}\right)\right\|_{L^2}\lesssim\epsilon^2 t^{-\frac74+C\epsilon}.
\]
\end{proof}

\bibliographystyle{abbrv}
\bibliography{KPRefs}

\begin{thebibliography}{10}

\bibitem{MR2358259}
J.~Colliander, A.~D. Ionescu, C.~E. Kenig, and G.~Staffilani.
\newblock Weighted low-regularity solutions of the {KP}-{I} initial-value
  problem.
\newblock {\em Discrete Contin. Dyn. Syst.}, 20(2):219--258, 2008.

\bibitem{MR2482120}
P.~Germain, N.~Masmoudi, and J.~Shatah.
\newblock Global solutions for 3{D} quadratic {S}chr\"odinger equations.
\newblock {\em Int. Math. Res. Not. IMRN}, (3):414--432, 2009.

\bibitem{GMS3dCapillary}
P.~{Germain}, N.~{Masmoudi}, and J.~{Shatah}.
\newblock {Global existence for capillary water waves}.
\newblock {\em ArXiv e-prints}, Oct. 2012.

\bibitem{MR2914945}
P.~Germain, N.~Masmoudi, and J.~Shatah.
\newblock Global solutions for 2{D} quadratic {S}chr\"odinger equations.
\newblock {\em J. Math. Pures Appl. (9)}, 97(5):505--543, 2012.

\bibitem{MR2993751}
P.~Germain, N.~Masmoudi, and J.~Shatah.
\newblock Global solutions for the gravity water waves equation in dimension 3.
\newblock {\em Ann. of Math. (2)}, 175(2):691--754, 2012.

\bibitem{MR2719895}
Z.~Guo, L.~Peng, and B.~Wang.
\newblock On the local regularity of the {KP}-{I} equation in anisotropic
  {S}obolev space.
\newblock {\em J. Math. Pures Appl. (9)}, 94(4):414--432, 2010.

\bibitem{MR2559713}
S.~Gustafson, K.~Nakanishi, and T.-P. Tsai.
\newblock Scattering theory for the {G}ross-{P}itaevskii equation in three
  dimensions.
\newblock {\em Commun. Contemp. Math.}, 11(4):657--707, 2009.

\bibitem{MR2526409}
M.~Hadac, S.~Herr, and H.~Koch.
\newblock Well-posedness and scattering for the {KP}-{II} equation in a
  critical space.
\newblock {\em Ann. Inst. H. Poincar\'e Anal. Non Lin\'eaire}, 26(3):917--941,
  2009.

\bibitem{hgmKdV}
B.~{Harrop-Griffiths}.
\newblock {Long time behavior of solutions to the mKdV}.
\newblock {\em ArXiv e-prints}, July 2014.

\bibitem{HayashiNaumkin2014}
N.~Hayashi and P.~I. Naumkin.
\newblock {L}arge {T}ime {A}symptotics for the {K}adomtsev-{P}etviashvili
  {E}quation.
\newblock {\em Comm. Math. Phys.}, to appear, 2014.

\bibitem{MR2976047}
N.~Hayashi, P.~I. Naumkin, and T.~Niizato.
\newblock Almost global existence of solutions to the
  {K}adomtsev-{P}etviashvili equations.
\newblock {\em Funkcial. Ekvac.}, 55(1):157--168, 2012.

\bibitem{MR1685890}
N.~Hayashi, P.~I. Naumkin, and J.-C. Saut.
\newblock Asymptotics for large time of global solutions to the generalized
  {K}adomtsev-{P}etviashvili equation.
\newblock {\em Comm. Math. Phys.}, 201(3):577--590, 1999.

\bibitem{itNLS}
M.~{Ifrim} and D.~{Tataru}.
\newblock {Global bounds for the cubic nonlinear Schr\"odinger equation (NLS)
  in one space dimension}.
\newblock {\em ArXiv e-prints}, Apr. 2014.

\bibitem{itCW}
M.~{Ifrim} and D.~{Tataru}.
\newblock {The lifespan of small data solutions in two dimensional capillary
  water waves}.
\newblock {\em ArXiv e-prints}, June 2014.

\bibitem{itWW}
M.~{Ifrim} and D.~{Tataru}.
\newblock {Two dimensional water waves in holomorphic coordinates II: global
  solutions}.
\newblock {\em ArXiv e-prints}, Apr. 2014.

\bibitem{MR2415308}
A.~D. Ionescu, C.~E. Kenig, and D.~Tataru.
\newblock Global well-posedness of the {KP}-{I} initial-value problem in the
  energy space.
\newblock {\em Invent. Math.}, 173(2):265--304, 2008.

\bibitem{MR1635416}
R.~J. I{\'o}rio, Jr. and W.~V.~L. Nunes.
\newblock On equations of {KP}-type.
\newblock {\em Proc. Roy. Soc. Edinburgh Sect. A}, 128(4):725--743, 1998.

\bibitem{KadomtsevPetviashvili}
B.~Kadomtsev and V.~Petviashvili.
\newblock On the stability of solitary waves in weakly dispersing media.
\newblock In {\em Sov. Phys. Dokl}, volume~15, pages 539--541, 1970.

\bibitem{MR2097033}
C.~E. Kenig.
\newblock On the local and global well-posedness theory for the {KP}-{I}
  equation.
\newblock {\em Ann. Inst. H. Poincar\'e Anal. Non Lin\'eaire}, 21(6):827--838,
  2004.

\bibitem{KSsurveyIST}
C.~{Klein} and J.-C. {Saut}.
\newblock {IST versus PDE, a comparative study}.
\newblock {\em ArXiv e-prints}, Sept. 2014.

\bibitem{MR1933858}
L.~Molinet, J.-C. Saut, and N.~Tzvetkov.
\newblock Global well-posedness for the {KP}-{I} equation.
\newblock {\em Math. Ann.}, 324(2):255--275, 2002.

\bibitem{MR1944575}
L.~Molinet, J.-C. Saut, and N.~Tzvetkov.
\newblock Well-posedness and ill-posedness results for the
  {K}adomtsev-{P}etviashvili-{I} equation.
\newblock {\em Duke Math. J.}, 115(2):353--384, 2002.

\bibitem{MR2047648}
L.~Molinet, J.-C. Saut, and N.~Tzvetkov.
\newblock Correction: ``{G}lobal well-posedness for the {KP}-{I} equation''
  [{M}ath. {A}nn. {\bf 324} (2002), no. 2, 255--275; mr1933858].
\newblock {\em Math. Ann.}, 328(4):707--710, 2004.

\bibitem{MR2830489}
T.~Niizato.
\newblock Large time behavior of solutions for the generalized
  {K}adomtsev-{P}etviashvili equation.
\newblock {\em Differ. Equ. Appl.}, 3(2):299--308, 2011.

\bibitem{MR803256}
J.~Shatah.
\newblock Normal forms and quadratic nonlinear {K}lein-{G}ordon equations.
\newblock {\em Comm. Pure Appl. Math.}, 38(5):685--696, 1985.

\end{thebibliography}

\end{document}